\documentclass[11pt]{article}
\usepackage{amssymb,latexsym}
\usepackage{tikz}
\usepackage{amsmath}
\usepackage{amsthm}
\usepackage{enumerate}
\newtheorem{theorem}{Theorem}[section]
\newtheorem{corollary}[theorem]{Corollary}

\newtheorem{examples}[theorem]{Examples}

\newtheorem{remark}[theorem]{Remark}

\usepackage{authblk}
\usepackage{tikz}

\providecommand{\keywords}[1]{\text{\textit{Key Words:}} #1}

\providecommand{\Classification}[1]{\text{\textit{2010 MSC:}} #1}

\begin{document}

\author{A. Cherrabi, H. Essannouni, E. Jabbouri and  A. Ouadfel\thanks{Corresponding Author: azzouz.cherrabi@gmail.com}}
 \affil{Laboratory of Mathematics, Computing and Applications-Information Security (LabMia-SI)\\
 Faculty of Sciences, Mohammed V University in Rabat.\\
 Rabat. Morocco.}
\title{On a new extension of the zero-divisor graph}
\date{}

\maketitle
\begin{abstract}In this paper, we introduce a new graph whose vertices are the nonzero zero-divisors of a commutative ring $R$, and for distincts elements $x$ and $y$ in the set $Z(R)^{\star}$ of the nonzero zero-divisors of $R$, $x$ and $y$ are adjacent if and only if $xy=0$ or $x+y\in Z(R)$. We present some properties and examples of this graph, and we study its relationship with the zero-divisor graph  and with a subgraph of the total graph of a commutative ring.
\end{abstract}


\Classification{Primary 13A15; Secondary 05C25}

\section*{Introduction}
Throughout this paper, $R$ will denote a commutative ring with $1\neq 0$, $Z(R)$  its set of zero-divisors,
 and $Z(R)^{\star}=Z(R)-\{0\}$.
The concept of a zero-divisor graph of a commutative ring was first introduced by I. Beck in \cite{E},
where he was  interested in colorings of commutative rings.
In I. Beck's work, all the elements of the commutative ring were vertices of the graph, and two distinct
elements $x$ and $y$ are adjacent if and only if $xy=0$. In \cite{D}, D.F. Anderson and P.S. Livingston
modified the definition by considering the graph $\Gamma (R)$ with vertices  the set of
nonzero zero-divisors of a ring $R$, and distinct vertices $x$ and $y$ are adjacent
if and only if $xy=0$. In \cite{C}, D.F. Anderson and A. Badawi introduced the total graph $T(\Gamma (R))$ of
a commutative ring $R$ with all elements of $R$ as vertices, and for distinct $x,y \in R$, the vertices $x$
and $y$ are adjacent if and only if $x+y \in Z(R)$. They also considered the induced subgraph $Z(\Gamma (R))$
with vertices in  $Z(R)$.

As stated in \cite{D},
the purpose of the introduction of the zero-divisor graph was to illuminate the structure of $Z (R)$.
Since  $Z (R)$ is a part of the ring $R$, we think that defining a graph using both operations, i.e.,
addition and multiplication, helps us to better approach the interactions between the elements of this semigroup.
In this context, we introduce, in this article, a new graph, denoted  $\widetilde{\Gamma}(R)$,
as the undirected simple graph  whose vertices are the nonzero zero-divisors of $R$, and for distinct
$x,y\in Z(R)^{\star}$, $x$ and $y$ are adjacent if and only if $xy=0$ or $x+y\in Z(R)$. It is obvious
that $\widetilde{\Gamma}(R)$ is just the union of $\Gamma(R)$ and the induced subgraph
$Z^{\star}(\Gamma (R))$ of $T(\Gamma (R))$ with vertices in $Z(R)^{\star}$ and that $Z^{\star}(\Gamma (R))$
and $\Gamma(R)$ are  spanning subgraphs of $\widetilde{\Gamma}(R)$. We  also note that the vertices of the
subgraph $Z(\Gamma (R))$ studied in \cite{C} are in $Z(R)$ while the vertices of  $Z^{\star}(\Gamma (R))$
are in $Z(R)^{\star}$.

Recall that a simple graph $G=(V,E)$ is connected if there exists a path between any
two distinct vertices. For distinct vertices $x$ and $y$ of $G$, the distance $d(x,y)$ is the length
of a shortest  path connecting $x$ and $y$; if there is no such path, $d(x,y)=\infty$. The diameter
of $G$ is $diam(G)=\sup\{d(x,y)|x,y \in V\, \text{and}\, x\neq y \}$. The girth of $G$, denoted $g(G)$,
is the  length of a shortest cycle in $G$, and $g(G)= \infty$ if $G$ contains no cycles. $G$ is complete
if it is connected with diameter less or equal to one, and $K_n$ denotes the complete graph with $n$ vertices. A basic reference
for graph theory is \cite{F}.\\
As usual, $\mathbb{Z}_n$ denotes the integers modulo $n$. A general reference
 for commutative ring theory is \cite{A}.

In the first section, we give some examples to illustrate the difference between the
graphs $\Gamma(R)$, $Z^{\star}(\Gamma (R))$ and the new graph $\widetilde{\Gamma}(R)$. We also give all the
graphs $\widetilde{\Gamma}(R)$ with $|Z(R)^{\star}|\leq 9$. The second section is devoted to studying the diameter,
the girth of $\widetilde{\Gamma}(R)$, and the relation between this graph and the two others  spanning subgraphs.
In the third section, we complete this study by giving some other properties of our graph.

\section{Examples}

We start with examples that illustrate the difference between the three graphs.  This can be seen using the figures corresponding to the two cases  $R=\mathbb{Z}_6$ and $R=\mathbb{Z}_2\times \mathbb{Z}_4$.

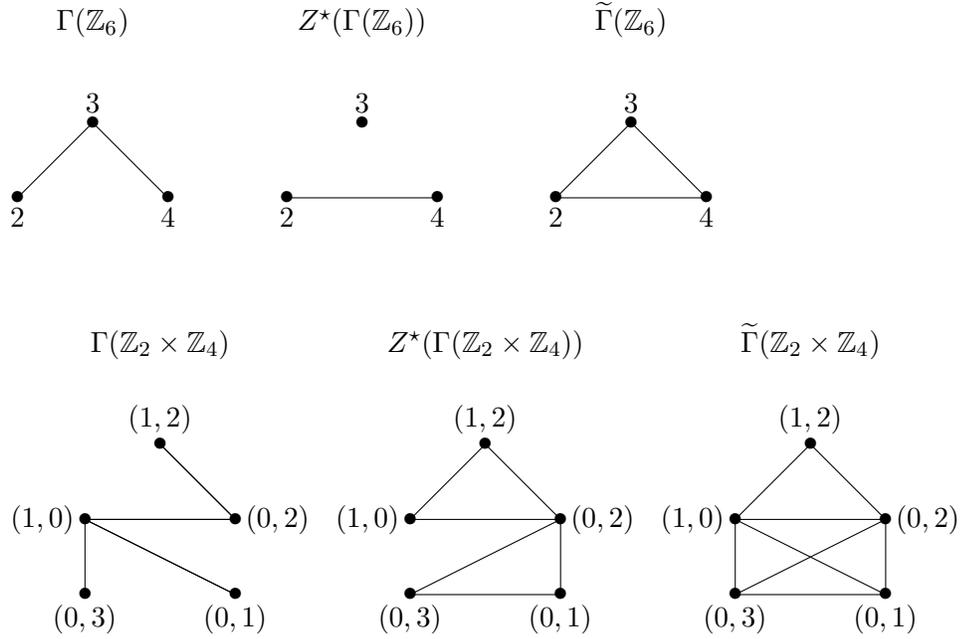
\begin{tikzpicture}
\draw (0,0)--(1,1);
\draw (1,1)--(2,0);
\draw (0,0) node[below]{$2$}node{$\bullet$};
\draw (1,1) node[above]{$3$}node{$\bullet$};
\draw (2,0) node[below]{$4$}node{$\bullet$};
\draw (1,2) node[above] {$\Gamma(\mathbb{Z}_6)$};
\end{tikzpicture}\hspace{1cm}
\begin{tikzpicture}
\draw (0,0)--(2,0);
\draw (0,0) node[below]{$2$}node{$\bullet$};
\draw (1,1) node[above]{$3$}node{$\bullet$};
\draw (2,0) node[below]{$4$}node{$\bullet$};
\draw (1,2) node[above] {$Z^{\star}(\Gamma(\mathbb{Z}_6))$};
\end{tikzpicture}\hspace{1cm}
\begin{tikzpicture}
\draw (0,0)--(2,0);
\draw (0,0)--(1,1);
\draw (1,1)--(2,0);
\draw (0,0) node[below]{$2$}node{$\bullet$};
\draw (1,1) node[above]{$3$}node{$\bullet$};
\draw (2,0) node[below]{$4$}node{$\bullet$};
\draw (1,2) node[above] {$\widetilde{\Gamma}(\mathbb{Z}_6)$};
\end{tikzpicture}

\vspace{1cm}

\begin{tikzpicture}
\draw (0,0)--(0,1);
\draw (1,2)--(2,1);
\draw (2,0)--(0,1);
\draw (2,1)--(0,1);
\draw (2,1)--(1,2);
\draw (0,1)--(2,0);
\draw (0,1)--(2,1);

\draw (0,0) node[below]{$(0,3)$} node{$\bullet$};
\draw (2,0) node[below]{$(0,1)$} node{$\bullet$};
\draw (0,1) node[left]{$(1,0)$}node{$\bullet$};
\draw (2,1) node[right]{$(0,2)$}node{$\bullet$};
\draw (1,2) node[above]{$(1,2)$}node{$\bullet$};
\draw (1,3) node[above] {$\Gamma(\mathbb{Z}_2\times \mathbb{Z}_4)$};
\end{tikzpicture}
\begin{tikzpicture}
\draw (2,0)--(2,1);
\draw (2,0)--(0,0);

\draw (0,1)--(2,1);

\draw (2,1)--(0,0);
\draw (2,1)--(1,2);

\draw (0,1)--(1,2);

\draw (0,0) node[below]{$(0,3)$} node{$\bullet$};
\draw (2,0) node[below]{$(0,1)$} node{$\bullet$};
\draw (0,1) node[left]{$(1,0)$}node{$\bullet$};
\draw (2,1) node[right]{$(0,2)$}node{$\bullet$};
\draw (1,2) node[above]{$(1,2)$}node{$\bullet$};
\draw (1,3) node[above] {$Z^{\star}(\Gamma(\mathbb{Z}_2\times \mathbb{Z}_4))$};
\end{tikzpicture}
\begin{tikzpicture}
\draw (0,1)--(1,2);
\draw (0,1)--(2,1)--(2,0)--(0,1);
\draw (0,0)--(2,1)--(1,2);
\draw (0,1)--(0,0)--(2,0);
\draw (0,0) node[below]{$(0,3)$} node{$\bullet$};
\draw (2,0) node[below]{$(0,1)$} node{$\bullet$};
\draw (0,1) node[left]{$(1,0)$}node{$\bullet$};
\draw (2,1) node[right]{$(0,2)$}node{$\bullet$};
\draw (1,2) node[above]{$(1,2)$}node{$\bullet$};
\draw (1,3) node[above] {$\widetilde{\Gamma}(\mathbb{Z}_2\times \mathbb{Z}_4$)};
\end{tikzpicture}
\smallskip

Based on the list given in \cite{G} of   commutative rings $R$ with $n\leq 9$, where $n=|Z(R)^{\star}|$, we see that:
  \begin{itemize}
    \item[--] If $n\in \{1,2,3,4,6,8\}$, $\widetilde{\Gamma}(R)$ is the complete graph $K_n$.
    \item[--] If $n\in \{5,7\}$, we obtain two non-isomorphic graphs, one of which is the complete graph $K_n$, and the other is not complete. For $n=5$, the non-complete graph is obtained by the non-isomorphic rings $\mathbb{Z}_2\times \mathbb{Z}_4$, and $\mathbb{Z}_2\times \mathbb{Z}_2[x]/(x^2)$ (figure 1). For $n=7$, the  non-complete graph is obtained by the non-isomorphic rings $\mathbb{Z}_{12}$ and $\mathbb{Z}_3\times  \mathbb{Z}_2[x]/(x^2)$ (figure 2).
    \item[--] If $n=9$, we obtain three non-isomorphic graphs, one of which is the complete graph $K_9$, and the others are not complete. The two non-complete graphs are obtained respectively by the  non-isomorphic rings  $\mathbb{Z}_4 \times \mathbb{F}_4$ and $\mathbb{Z}_2[x]/(x^2) \times \mathbb{F}_4$ (figure 3), and by $\mathbb{Z}_2\times \mathbb{Z}_2\times \mathbb{Z}_3$ (figure 4).
  \end{itemize}

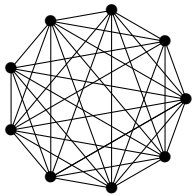
\begin{center}
  Figure 1: the only non-complete graph obtained when $|Z(R)^{\star}|=5$.
  \begin{tikzpicture}
\draw (0,1)--(1,2);
\draw (0,1)--(2,1)--(2,0)--(0,1);
\draw (0,0)--(2,1)--(1,2);
\draw (0,1)--(0,0)--(2,0);
\draw (0,0)  node{$\bullet$};
\draw (2,0)  node{$\bullet$};
\draw (0,1) node{$\bullet$};
\draw (2,1) node{$\bullet$};
\draw (1,2) node{$\bullet$};
\end{tikzpicture}
\end{center}

\begin{center}
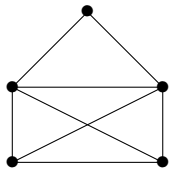
  Figure 2: the only non-complete graph obtained when $|Z(R)^{\star}|=7$.
  \begin{tikzpicture}
   \newdimen\R
   \R=1cm
   \draw (0:\R)
      \foreach \x in {52,104,...,360} {  -- (\x:\R) }
               -- cycle (360:\R) node{$\bullet$}  
               -- cycle (308:\R)  node{$\bullet$} 
               -- cycle (256:\R) node{$\bullet$} 
               -- cycle (204:\R) node{$\bullet$} 
               -- cycle  (152:\R) node{$\bullet$}
               -- cycle  (101:\R) node{$\bullet$} 
               -- cycle  (51:\R) node{$\bullet$}  ; 

      \draw (204:\R)-- (101:\R);
      \draw (204:\R)-- (51:\R);

      \draw (360:\R)-- (256:\R);
      \draw (360:\R)-- (101:\R);

      \draw (256:\R)-- (51:\R);
      \draw (256:\R)-- (101:\R);
      \draw (256:\R)-- (308:\R);
      \draw (256:\R)--(152:\R);

      \draw (256:\R)--(101:\R);
      \draw (256:\R)--(308:\R);
      \draw (256:\R)--(152:\R);
      \draw (256:\R)--(308:\R);

      \draw (101:\R)--(308:\R);
      \draw (101:\R)--(152:\R);

      \draw (308:\R)--(51:\R);

      \draw (152:\R)--(51:\R);
\end{tikzpicture}
\end{center}

\begin{center}
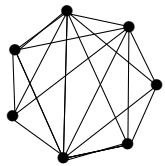
  Figure 3: the first non-complete graph obtained when $|Z(R)^{\star}|=9$.
  \begin{tikzpicture}
   \newdimen\R
   \R=1.2cm
   \draw (0:\R)
      \foreach \x in {40,80,...,360} {  -- (\x:\R) }
               -- cycle (360:\R) node{$\bullet$}  
               -- cycle (320:\R)  node{$\bullet$}
               -- cycle (280:\R) node{$\bullet$} 
               -- cycle (240:\R) node{$\bullet$} 
               -- cycle  (200:\R) node{$\bullet$} 
               -- cycle  (160:\R) node{$\bullet$} 
               -- cycle  (120:\R) node{$\bullet$}  
               -- cycle  (80:\R) node{$\bullet$}   
               -- cycle  (40:\R) node{$\bullet$};   

      \draw (360:\R)--(80:\R);
      \draw (360:\R)--(320:\R);
      \draw (360:\R)--(280:\R);
      \draw (360:\R)--(240:\R);

      \draw (240:\R)--(320:\R); 
      \draw (240:\R)--(360:\R);
      \draw (240:\R)--(40:\R);
      \draw (240:\R)--(80:\R);
      \draw (240:\R)-- (120:\R);
      \draw (240:\R)--(160:\R);

      \draw (280:\R)--(40:\R); 
      \draw (280:\R)--(80:\R);
      \draw (280:\R)--(120:\R);
      \draw (280:\R)--(160:\R);
      \draw (280:\R)-- (200:\R);

      \draw (320:\R)--(40:\R); 
      \draw (320:\R)--(80:\R);
      \draw (320:\R)--(120:\R);
      \draw (320:\R)--(160:\R);
      \draw (320:\R)-- (200:\R);

      \draw (80:\R)--(160:\R); 
      \draw (80:\R)--(200:\R);

      \draw (120:\R)--(200:\R); 
\end{tikzpicture}
\end{center}

\begin{center}
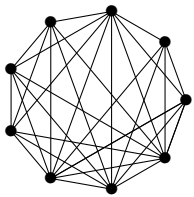
  Figure 4: the second non-complete graph obtained when $|Z(R)^{\star}|=9$.
  \begin{tikzpicture}
   \newdimen\R
   \R=1.2cm
   \draw (0:\R)
      \foreach \x in {40,80,...,360} {  -- (\x:\R) }
               -- cycle (360:\R) node{$\bullet$}  
               -- cycle (320:\R)  node{$\bullet$} 
               -- cycle (280:\R) node{$\bullet$} 
               -- cycle (240:\R) node{$\bullet$} 
               -- cycle  (200:\R) node{$\bullet$}
               -- cycle  (160:\R) node{$\bullet$} 
               -- cycle  (120:\R) node{$\bullet$}  
               -- cycle  (80:\R) node{$\bullet$}   
               -- cycle  (40:\R) node{$\bullet$} ;  

      \draw (360:\R)--(80:\R);
      \draw (360:\R)--(120:\R);
      \draw (360:\R)--(160:\R);
      \draw (360:\R)--(200:\R);
      \draw (360:\R)--(240:\R);
      \draw (360:\R)--(280:\R);

      \draw (240:\R)--(320:\R); 
      \draw (240:\R)--(360:\R);
      \draw (240:\R)--(80:\R);
      \draw (240:\R)-- (120:\R);
      \draw (240:\R)--(160:\R);

      \draw (280:\R)--(40:\R); 
      \draw (280:\R)--(80:\R);
      \draw (280:\R)--(120:\R);
      \draw (280:\R)--(160:\R);
      \draw (280:\R)-- (200:\R);

      \draw (320:\R)--(40:\R); 
      \draw (320:\R)--(80:\R);
      \draw (320:\R)--(120:\R);
      \draw (320:\R)--(160:\R);
      \draw (320:\R)-- (200:\R);

      \draw (80:\R)--(200:\R); 

      \draw (40:\R)--(120:\R); 
      \draw (40:\R)--(160:\R);
      \draw (40:\R)--(200:\R);

      \draw (120:\R)--(200:\R); 
\end{tikzpicture}
\end{center}

\section{Diameter, girth and the relation between $\widetilde{\Gamma}(R)$ and the others graphs}

Since $\Gamma(R)$ is connected with $diam(\Gamma(R))\leq 3$ (see \cite{D}) and $\Gamma(R)$ is a spanning subgraph of $\widetilde{\Gamma}(R)$, it follows that $\widetilde{\Gamma}(R)$ is also connected with diameter at most $3$. On the other hand, in general, $Z^{\star}(\Gamma(R))$ is a non-connected spanning subgraph of $\widetilde{\Gamma}(R)$ (see the example  where $R= \mathbb{Z}_6$ cited in Section 1). However,  the following result shows that $diam(\widetilde{\Gamma}(R))\leq 2$.

\begin{theorem}If $x$ and $y$ are distinct vertices of $\widetilde{\Gamma}(R)$, then $d(x,y)\leq 2$. Consequently, the graph $\widetilde{\Gamma}(R)$ is connected and $diam(\widetilde{\Gamma}(R))\leq 2$.
\end{theorem}

\begin{proof}Let $x,y\in Z(R)^{\star}$; so there exists $a,b\in Z(R)^{\star}$ such that $ax=0$ and $by=0$. If $ab\neq 0$, then $ab(x+y)=0$, and so $x$ and $y$ are adjacent. If $ab=0$, then $a(b+x)=0$, so $x-b-y$.
\end{proof}

Notice that the diameter of $\widetilde{\Gamma}(R)$ can be $0$ (case where $R=\mathbb{Z}_4$). Also, the examples given in Section 1 show that the diameter of $\widetilde{\Gamma}(R)$ can be one (case where $R=\mathbb{Z}_6$) or two (case where $R=\mathbb{Z}_2\times \mathbb{Z}_4$).\medskip

 We know that $g(\Gamma(R))\in \{3,4,\infty\}$ (see \cite {D}), and, in general, $g(Z^{\star}(\Gamma(R)))=\infty$. However, in the following result, we show that  $g(\widetilde{\Gamma}(R))=3$.

\begin{theorem}If $|Z(R)^{\star}| \geq 3$, then $g(\widetilde{\Gamma}(R))=3$.
\end{theorem}

\begin{proof}If $\widetilde{\Gamma}(R)$ is complete, then $g(\widetilde{\Gamma}(R))=3$. Suppose that there exists two distinct vertices $x,y$ such that $x$ and $y$ are not connected. Then there exist $a,b\in Z(R)^{\star}$ with $ax=by=0$. Then $ab(x+y)=0$, and hence $ab=0$, because $x+y\notin Z(R)$. Thus $a(b+x)=0$; therefore $b+x\in Z(R)$. On the other hand, we have $x\neq a$,  if not $bx=ba=0$; so $b(x+y)=0$, and hence $x$ and $y$ are adjacent. Then $a\neq b$. If not, then $ay=by=0$; so $a(x+y)=0$. Then $x$ and $y$ are adjacent. Also,  $x\neq b$. If not, then $xy=by=0$; so  $x$ and $y$ are adjacent. Thus we obtain a $3$-cycle $x-a-b-x$.
\end{proof}
 
Given that the diameter can be zero, one, or two, we will characterize the finite rings $R$ such that the graph $\widetilde{\Gamma}(R)$ is a complete.

\begin{theorem} Let $R$ be a finite ring. Then $\widetilde{\Gamma}(R)$ is a complete graph if and only if $R$ is a local ring, a product of two fields, or $R \cong (\mathbb{Z}_2)^n$.
\end{theorem}
\begin{proof}
$(\Leftarrow )$ It is obvious that if $R$ is a local ring or a product of two fields, then $\widetilde{\Gamma}(R)$ is a complete graph. Suppose that $R=\mathbb{Z}_2^n$; so $R$ is a boolean ring. Let $x,y\in Z(R)^{\star}$. If $xy\neq 0$, then $xy(x+y)=xy+xy=0$; so $x+y\in Z(R)$.\\
$(\Rightarrow )$ Suppose that $R$ is not local. Since $R$ is finite, $R\cong R_1\times \cdots \times R_n$ and $n\geq 2$, where $R_1,...,  R_n$ are local rings.\\
We claim that all the rings $R_i$ are  fields. If not, there exists an $i$ such that $R_i$ is not a field; so there exists $x_i\in Z(R_{i})^{\star}$. Then consider $x=(1,\dots,x_i,\dots,1), y=(0,\dots,1,\dots 0)\in Z(R)^{\star}$; so $xy\neq 0$ and $x+y\notin Z(R)$. If not, then $1\in \mathfrak{m}_i$, where $\mathfrak{m}_i$ is the maximal ideal of $R_i$.\\
Suppose that $n>2$. Then $char(R)=2$. If not, then consider $x=(0,1, \dots , 1)$, $y=(1,0,1, \dots ,1)$. Then $xy\neq 0$ and $x+y\notin Z(R)$. To finish, we claim that $R\cong \mathbb{Z}_2^n$. If not, there exists an $i$ such that $R_i\not \cong \mathbb{Z}_2$; so there exists $x_i\in R_i$ such that $x_i\notin \{0,1\}$. Then we consider $x=(0,\dots ,x_i,1,\dots,1), y=(1,\dots,x_i+1, 0,\dots,0)$. Hence we obtain $xy\neq 0$ because $x_i(x_i+1)\neq 0$, and $x+y=(1,\dots, 1)\notin Z(R)$ because $2x_i=0$.
\end{proof}

\begin{corollary}Let $R$ be a reduced finite ring such that $R$ is not a field and $R$ is not boolean. Then $\widetilde{\Gamma}(R)$ is a complete graph if and only if $R$ has exactly two maximal ideals.
\end{corollary}

\begin{proof}Since $R$ is a reduced  finite ring, $R\cong K_1\times \dots \times K_n$, where $K_i=R/\mathfrak{m}_i$ and $\mathfrak{m}_1,\dots, \mathfrak{m}_n$ are the maximal ideals of $R$. So $\widetilde{\Gamma}(R)\cong \widetilde{\Gamma}(K_1\times \dots \times K_n)$ as graphs. Then, since $R$ is not boolean, $\widetilde{\Gamma}(R)$ is complete if and only if $n=2$.
\end{proof}

It is easy to deduce the following corollary from the preceding theorem.

\begin{corollary} Let $n>1$ be a composite integer. Then $\widetilde{\Gamma}(\mathbb{Z}_n)$ is complete if and only if $n$ is a power of a prime number or $n$ is a product of two distinct primes.
\end{corollary}

The following corollary determines when  the complete graph  $K_n$ is obtained using the graph $\widetilde{\Gamma}(\mathbb{Z}_k)$.

\begin{corollary} Let $n\in \mathbb{N}^{\star}$.
\begin{enumerate}[(a)]
  \item Suppose that $n$ is even. Then there exists $k\in \mathbb{N}$ such that $\widetilde{\Gamma}(\mathbb{Z}_k)=K_n$ if and only if $n+1$ is a power of a prime number or the even number $n+2$ satisfies the Goldbach's conjecture with two distinct prime numbers.
  \item Suppose that $n$ is odd. Then there exists $k\in \mathbb{N}$ such that $\widetilde{\Gamma}(\mathbb{Z}_k)=K_n$ if and only if $n+1$ is a power of $2$ or $n$ is prime.
\end{enumerate}
\end{corollary}

\begin{proof}
(a) Suppose that $\widetilde{\Gamma}(\mathbb{Z}_k)=K_n$. Then $k=p^{\alpha}$ or $k=pq$, where $p$ and $q$ are distinct primes. If $k=p^{\alpha}$, then $|Z(\mathbb{Z}_k)^{\star}|=p^{\alpha -1}-1=n$, and $n+1=p^{\alpha -1}$.
If $k=pq$, then $|Z(\mathbb{Z}_k)^{\star}|=p+q-2=n$, and $n+2=p+q$.\\
Conversely, if $n+1=p^{\alpha}$, let $k=p^{\alpha+1}$. Then $\widetilde{\Gamma}(\mathbb{Z}_k)$ is complete and $|Z(\mathbb{Z}_k)^{\star}|=p^{\alpha}-1=n$. If $n+2=p+q$, let $k=pq$. Then $\widetilde{\Gamma}(\mathbb{Z}_k)$ is complete and $|Z(\mathbb{Z}_k)^{\star}|=n$.\\
(b) Suppose that $\widetilde{\Gamma}(\mathbb{Z}_k)=K_n$. Then $k=p^{\alpha}$ or $k=pq$, where $p$ and $q$ are distinct primes. If $k=p^{\alpha}$, then $|Z(\mathbb{Z}_k)^{\star}|=p^{\alpha -1}-1=n$. Then $n+1=p^{\alpha -1}$ and $p=2$ because $n+1$ is even.
If $k=pq$, with $p<q$.  Then $|Z(\mathbb{Z}_k)^{\star}|=p+q-2=n$, and $n+2=p+q$. Therefore $p=2$ because $n$ is odd. Thus $n=q$.
Conversely, if $n+1=2^{\alpha}$ or $n=q$,  let   $k=2^{\alpha+1}$ and  $k=2q$, respectively, where $q$ is an odd prime.
\end{proof}

\begin{examples}\hfill
\begin{itemize}
  \item[--] To get $K_{12}$,  take $k=13^2=169$. Then $\widetilde{\Gamma}(\mathbb{Z}_{169})=K_{12}$. Also, to get $K_{48}$,  take $k=7^3=343$, $k=3\cdot47=141$, $k=7\cdot 43=301$,  $k=13\cdot 37=481$, or $k=19\cdot 31=589$, and then $\widetilde{\Gamma}(\mathbb{Z}_{k})=K_{48}$.
  \item[--] To get $K_7$,  take $k=16$. Then $\widetilde{\Gamma}(\mathbb{Z}_{16})=K_7$. Also, to get $K_{11}$,  take $k=22$. Then $\widetilde{\Gamma}(\mathbb{Z}_{22})=K_{11}$.
  \item[--] There is no $k\in \mathbb{N}$ with $\widetilde{\Gamma}(\mathbb{Z}_{k})=K_{9}$.
\end{itemize}
\end{examples}

In the rest of this section, we determine the cases  when $\widetilde{\Gamma}(R)$ coincides with the spanning subgraphs $\Gamma(R)$ and $Z^{\star}(\Gamma(R))$.

\begin{theorem}Suppose that $R$ is finite. Then $\Gamma(R)=\widetilde{\Gamma}(R)$ if and only if $\Gamma(R)$ is complete.
\end{theorem}

\begin{proof}It is obvious that if $R$ is a field, then $\Gamma(R)=\widetilde{\Gamma}(R)$; so we suppose that $R$ is not a field. Since $\Gamma(R)$ is a spanning subgraph of $\widetilde{\Gamma}(R)$,  if $\Gamma(R)$ is a complete graph, then $\widetilde{\Gamma}(R)$ is too. So $\Gamma(R)=\widetilde{\Gamma}(R)$.\\
Conversely, since $R$ is a finite ring, $R\cong R_1\times \cdots \times R_n$, where $ R_1, \dots, R_n$ are local rings. Then $n\leq 2$. Indeed, if $n>2$, then $\Gamma(R)\neq \widetilde{\Gamma}(R)$ because $(0,1,0,\dots,0)$ and $(0,1,1,0,\dots,0)$ are adjacent in $\widetilde{\Gamma}(R)$ and are not adjacent in $\Gamma(R)$. If $n=1$, then $R$ is local; so $\Gamma(R)=\widetilde{\Gamma}(R)$ is complete. If $n=2$, we  also claim that $R_1$ and $R_2$ are fields because if, for example, $R_1$ is not a field, then $(x,1)$ and $(0,1)$, where $x\in Z(R_1)^{\star}$, are adjacent in $\widetilde{\Gamma}(R)$, but they are not adjacent in $\Gamma(R)$. Since $R_1$ and $R_2$ are fields, $\Gamma(R)=\widetilde{\Gamma}(R)$ is complete by Theorem 2.4.
\end{proof}

Using the previous theorem and Theorem $3.3 $ of \cite{D}, we obtain.

\begin{corollary}Let $R$ be a finite ring. If $\Gamma(R)=\widetilde{\Gamma}(R)$, then either $R\cong \mathbb{Z}_2\times \mathbb{Z}_2$ or $(R,\mathfrak{m})$ is local with $\mathfrak{m}^2=0$, $char(R)=p$ or $p^2$, and $|\Gamma(R)|=p^n-1$, where $p$ is prime and $n\geq 1$.
\end{corollary}

\begin{corollary} Let $n>1$ be a composite integer. Then $\Gamma(\mathbb{Z}_n)=\widetilde{\Gamma}(\mathbb{Z}_n)$ if and only if $n=p^2$, where $p$ is prime.
\end{corollary}

The following result determines when $\widetilde{\Gamma}(R)$ coincides with the other graph $Z^{\star}(\Gamma(R))$.

\begin{theorem}Let $R$ be a finite ring. Then $\widetilde{\Gamma}(R)=Z^{\star}(\Gamma(R))$ if and only if $R$ is local.
\end{theorem}

\begin{proof}It is obvious that if $R$ is a field, then $Z^{\star}(\Gamma(R))=\widetilde{\Gamma}(R)$. So we suppose that $R$ is not a field. If $R$ is local, then  $Z^{\star}(\Gamma(R))$ is  complete; so $\widetilde{\Gamma}(R)$ is complete. Conversely, since $R$ is finite, $R\cong R_1\times\cdots\times R_n$, where $R_1,\dots, R_n$ are local rings. We claim that $n=1$. Indeed, if $n>1$, then consider $x=(1,0,1,\dots,1)$ and $y=(0,1,0,\dots,0)$. It is obvious that $x,y\in Z(R)^{\star}$ are adjacent as vertices of $\widetilde{\Gamma}(R)$, but they are not adjacent as vertices of  $Z^{\star}(\Gamma(R))$. Hence $\widetilde{\Gamma}(R)\neq Z^{\star}(\Gamma(R))$.
\end{proof}

\begin{corollary} $\widetilde{\Gamma}(R)=Z^{\star}(\Gamma(R))$ if and only if $Z(R)$ is an ideal of $R$, if and only if $Z^{\star}(\Gamma(R))$ is complete.
\end{corollary}

\begin{corollary}Let $n>1$ be a composite integer. Then $\widetilde{\Gamma}(\mathbb{Z}_n)=Z^{\star}(\Gamma(\mathbb{Z}_n))$ if and only if $n$ is a power of a prime.
\end{corollary}

\begin{remark}It is obvious that if $\Gamma (R)$  or  $Z^{\star}(\Gamma(R))$ is complete, then $\widetilde{\Gamma}(R)$ is complete. However, the converse is false as shown in the example where $R =\mathbb{Z}_6$.
\end{remark}

\section{Others Properties of $\widetilde{\Gamma}(R)$}

In this section, we give some more properties of $\widetilde{\Gamma}(R)$.

\begin{theorem} If $R$ is a finite ring which is not a field, then there exists  a vertex of $\widetilde{\Gamma}(R)$ which is adjacent to every other vertex.
\end{theorem}

\begin{proof}Since $R$ is finite, $R\cong R_1\times\cdots \times R_n$, where $R_1,\dots, R_n$ are  finite local rings.\\
Case 1: Suppose that all the $R_i$ are fields. If $n=2$, then $\widetilde{\Gamma}(R)$ is complete, and the result is obvious. If $n\geq 3$, let $a=(a_1,0,\dots, 0)\in Z(R)^{\star}$. So for every $x=(x_1,\dots, x_n)\in Z(R)^{\star}$, $a$ and $x$ are adjacent. Indeed, since $x=(x_1,\dots, x_n)$ is in $Z(R)^{\star}$, there exists an $i$ such that $x_i=0$. If $i=1$, then $ax=0$. If $i\neq 1$, then $a+x=(a_1+x_1,\dots, 0,\dots, x_n)\in Z(R)$.\\
Case 2: Suppose that there exists an $i$ such that $R_i$ is not a field. Thus we can suppose that $R_1$ is not a field. Let $a_1\in Z(R_1)^{\star}$ and $a=(a_1,0,\dots, 0)\in R$; we have $a\in  Z(R)^{\star}$. Let $x=(x_1,\dots, x_n)\in Z(R)^{\star}$. If there exists an $i$ such that $x_i=0$, we check as before that $a$ and $x$ are adjacent. If for every $ i$, $x_i\neq 0$, so there exists $j$ such that $x_j\in Z(R_j)^{\star}$. If $j=1$, then $a+x=(a_1+x_1,x_2,\dots, x_n)\in Z(R)$ because $a_1+x_1\in Z(R_1)$, and if $j\neq 1$, then $a+x=(a_1+x_1,x_2,\dots, x_j,\dots, x_n)\in Z(R)$ because $x_j\in Z(R_j)$.
\end{proof}

\begin{remark}
According to the previous theorem, if $R$ is finite, then $\widetilde{\Gamma}(R)$ has a spanning (non-induced) subgraph which is a star graph. In general, the graphs $\Gamma(R)$ and $Z^{\star}(\Gamma(R))$ do not have this property.
\end{remark}

\begin{theorem}\label{thm.tri}\hfill
\begin{enumerate}[(a)]
  \item If $|Z(R)^{\star}|\geq 3$, then every vertex  of $\widetilde{\Gamma}(R)$ is a vertex in a triangle.
  \item If $x$ and $y$ are two distinct, non-adjacent vertices, then there are vertices $z$ and $t$ such that $x-z-y-t$ is a square.
\end{enumerate}
\end{theorem}
\begin{proof}(a) Let $x\in Z(R)^{\star}$.\\
Case 1: Suppose there is $y\in Z(R)^{\star}$ such that $x$ and $y$ are not adjacent. Then, with the notation used in the proof of Theorem 2.2, we have $x-a-b-x$ is a triangle.\\
Case 2: Suppose that, for every $  t\in Z(R)^{\star}$, $x$ and $t$ are adjacent. Let $y,z\in Z(R)^{\star}$ such that $x,y$, and $z$ are distinct. We can suppose that $y$ and $z$ are not adjacent. If not, it is obvious that $x-y-z-x$ is a triangle; so $yz\neq 0$ and $y+z\notin Z(R)$. We have $xy\neq 0$ or $xz\neq 0$. If not, then $x(y+z)=0$;  so $y$ and $z$ are adjacent. Suppose that $xy\neq 0$ (the case $xz\neq 0$ is similar). Let $a,b\in Z(R)^{\star}$ such that $ay=bz=0$ and $s=y+z$. Since $y$ and $z$ are not adjacent, $s\notin Z(R)$. So  $sb(a+z)=sba=(y+z)ba=yba=0$, and since $y,z$ are not adjacent $sb=(y+z)b\neq 0$. Hence $a+z\in Z(R)$. Since $xy\neq 0$, $a\neq x$ and since $yz\neq 0$, $a\neq z$, then $x-a-z-x$  is a triangle.

(b) Let $x,y\in Z(R)^{\star}$ such that $x\neq y$, $xy\neq 0$, and $x+y\notin Z(R)$. So there are $z,t\in Z(R)^{\star}$ such that $xz=0$ and $yt=0$. Let $s=x+y$. Then $st(y+z)=stz=t(sz)=t(yz)=0$, with $st\neq 0$ because $s\notin Z(R)$. Thus $y+z\in Z(R)$. In the same way, $sz(x+t)=szt=z(st)=z(xt)=0$. Thus $x+t\in Z(R)$. On the other hand, the four vertices $x,y,z$ and $t$ are distinct. Indeed, $x\neq t$. If not, then $x$ and $y$ are adjacent. And $x\neq z$. If not, then $x$ and $y$ are adjacent. And $y\neq t$. If not, then $x$ and $y$ are adjacent. And $t\neq z$. If not, $z(x+y)=0$, then $x$ and $y$ are adjacent. Finally,  $y\neq z$. If not, then $xy=0$. Thus $x-z-y-t$ is a square.
\end{proof}

 As it can be seen, the properties of Theorem \ref{thm.tri} are satisfied for either  finite or infinite rings.\medskip

Following \cite{H}, a graph $G=(V,E)$ is said to be hypotriangulated if for every
2-path $x-z-y$, $x$ and $y$ are adjacent or there exists a vertex $t\neq z$ such that $x-t-y$ is a 2-path. Thus, from Theorem \ref{thm.tri}, $\widetilde{\Gamma}(R)$ is a hypotriangulated graph.\\

We end this paper with the following remark.

\begin{remark}   We recall that a chord of a cycle $C=(V(C),E(C))$ in a graph $G=(V,E)$ is an edge in $E\setminus E(C)$ such that both of whose ends lie on $V(C)$. Also, a graph is triangulated (or chordal) if every cycle of length greater than three has a chord. We ask if $\widetilde{\Gamma}(R)$ is triangulated. The answer is negative as shown by the following example. Let $n=2\cdot 3\cdot 5\cdot 13=390$. It is obvious that  $C=(2-3-5-8-2)$ is a cycle in $\widetilde{\Gamma}(\mathbb{Z}_n)$ and $C$ does not have a chord.
\end{remark}

\bigskip

\noindent {\bf Acknowledgement.} \\
We would like  to thank the referee and  Driss Bennis for careful reading and helpful comments.


\end{document}